\documentclass{amsart}
\usepackage{amssymb}

\newtheorem{thm}{Theorem}

\newtheorem{lem}{Lemma}

\theoremstyle{definition}

\newcommand{\A}{{\mathcal A}}
\newcommand{\R}{{\mathcal R}}

\newcommand{\es}{{\mathcal S}}

\newcommand{\IC}{{\mathbb C}}
\newcommand{\ID}{{\mathbb D}}

\newcommand{\CC}{{\mathcal C}}

\newcommand{\D}{{\mathbb D}}

\newcommand{\real}{{\operatorname{Re}\,}}

\def\be{\begin{equation}}
\def\ee{\end{equation}}

\begin{document}
\title[Improved bounds of the third Hankel determinant]{Improved bounds of the third Hankel determinant for classes of univalent functions with bounded turning}

\author[M. Obradovi\'{c}]{Milutin Obradovi\'{c}}
\address{Department of Mathematics,
Faculty of Civil Engineering, University of Belgrade,
Bulevar Kralja Aleksandra 73, 11000, Belgrade, Serbia}
\email{obrad@grf.bg.ac.rs}
\author[N. Tuneski]{Nikola Tuneski}
\address{Department of Mathematics and Informatics, Faculty of Mechanical Engineering, Ss. Cyril and Methodius
University in Skopje, Karpo\v{s} II b.b., 1000 Skopje, Republic of North Macedonia.}
\email{nikola.tuneski@mf.edu.mk}
\author[P. Zaprawa]{Pawe{\l} Zaprawa}
\address{Faculty of Mechanical Engineering, Lublin University of Technology, Nadbystrzycka 36, 20-618 Lublin, Poland.}
\email{nikola.tuneski@mf.edu.mk}

%\author{S. Ponnusamy${}^{~\mathbf{*}}$}
%\address{S. Ponnusamy, Department of Mathematics,
%Indian Institute of Technology Madras, Chennai--600 036, India.}
%\email{samy@iitm.ac.in}
%\author{N. Tuneski}
%\address{N. Tuneski, St. Cyril and Methodius University, Faculty of Mechanical
%Engineering, Karpo\v s II b.b., 1000 Skopje, R. Macedonia.}
%\email{nikolat@mf.edu.mk}


\begin{abstract}
In this paper we improve the bounds of the third order Hankel determinant for two classes of univalent functions with bounded turning.
\end{abstract}

%\thanks{The work of the first author was supported by MNZZS Grant, No. ON174017, Serbia. The research of the second
%author was supported by National Board for Higher
%Mathematics, India.}

\maketitle

\section{Introduction and preliminaries}

\medskip

Univalent functions, that are functions which are analytic, one-on-one and onto a certain domain, play a significant role in geometric function theory and in complex analysis in general. In spite the main problem in the area, the Bieberbach conjecture, was closed by de Branges in 1984 the theory of univalent functions still remains attractive. A concept from this theory that was recently rediscovered and finds its application in the theory of singularities (see \cite{dienes}) and in the study of power series with integral coefficients is the Hankel determinant of functions $f(z)=z+a_2z^2+a_3z^3+\cdots$ analytic in the unit disk $\ID := \{ z\in \IC:\, |z| < 1 \}$, for $q\geq 1$ and $n\geq 1$ defined by
\[
        H_{q}(n) = \left |
        \begin{array}{cccc}
        a_{n} & a_{n+1}& \ldots& a_{n+q-1}\\
        a_{n+1}&a_{n+2}& \ldots& a_{n+q}\\
        \vdots&\vdots&~&\vdots \\
        a_{n+q-1}& a_{n+q}&\ldots&a_{n+2q-2}\\
        \end{array}
        \right |.
\]
The class of all such functions is denoted by $\A$.

\medskip

The upper bound (preferebly sharp) of the modulus  of the Hankel determinants is extensively studied in recent time, mainly the second order case, $H_{2}(2)= a_2a_4-a_{3}^2$, and the third order case,
\[ H_3(1) =  \left |
        \begin{array}{ccc}
        1 & a_2& a_3\\
        a_2 & a_3& a_4\\
        a_3 & a_4& a_5\\
        \end{array}
        \right | = a_3(a_2a_4-a_{3}^2)-a_4(a_4-a_2a_3)+a_5(a_3-a_2^2).
\]

\medskip

This problem, as most others over the class of univalent functions is difficult to tackle with for the general class, and instead its subclasses are studied. The best known result for the whole class is the one of Hayman (\cite{hayman-68}) who showed that $|H_2(n)|\le An^{1/2}$, where $A$ is an absolute constant, and that this rate of growth is the best possible. For the subclasses, we list the results for the classes of starlike and convex functions
\[ \mathcal{S}^\ast = \left\{ f\in\A:\real\frac{zf'(z)}{f(z)}>0,\, z\in\D\right\}  \]
and
\[ \mathcal{C} = \left\{ f\in\A:\real\left[1+\frac{zf''(z)}{f'(z)}\right]>0,\,0 z\in\D\right\},  \]
with the upper bound of the second Hankel determinant  1 and $1/8$ (\cite{janteng-07}), and  of the third Hankel determinant $0.777987\ldots$ (\cite{MONT-2019-3}) and $\frac{4}{135}=0.0296\ldots$ (\cite{Kowalczyk-18}), respectively. The estimates for the second order determinant are sharp, while of the third order are not, but are best known.

\medskip

Other related results can be found in \cite{ind-1,Kowalczyk-18,MONT-2018-1,DTV-book,krishna-1}.

\medskip

In this paper we will study the class $\R\subset\A$ of univalent  functions satisfying
\be\label{e1}
\real f'(z)>0 \quad\quad (z\in\D),
\ee
and the class $\R_1\subset\A$ satisfying
\[  \real \{f'(z)+zf''(z)\}>0 \quad\quad (z\in\D). \]
The functions from the class $\R$ are said to be of bounded turning since $\real f'(z)>0$ is equivalent to $|\arg f'(z)|<\pi/2$, and $\arg f(z)$ is the angle of rotation of the image of a line segment starting from $z$ under the mapping $f$. They are of special interest since they are not part of class of starlike  functions which is very wide subclass of univalent functions. This is due to the counterexample by Krzy\.z (\cite{krzyz}) showing that $\es^\ast$ does not contain $\R$, and $\R$ does not contain $\es^\ast$.
In addition,  classes $\R$ and $\R_1$ are related in the same way as the classes of starlike and convex functions, i.e., $\R_1\subset \R$ (\cite{ali}) as $\CC\subset\es^\ast$, and
\[ f\in\R_1 \quad\quad \Leftrightarrow \quad\quad zf'(z)\in\R, \]
as
\[ f\in \CC\quad\quad \Leftrightarrow\quad\quad  zf'(z)\in\es^\ast . \]

\medskip

For the class $\R$ in \cite{janteng-06} the authors showed that
\[|H_2(2)|\le \frac{4}{9}  = 0.444\ldots,\]
and in \cite{khatter} (with $\alpha=1$ in Corollary 2.8),
\[|H_3(1)|\le \frac{1}{540} \left(\frac{877}{3}+25 \sqrt{5}\right) = 0.64488\ldots.\]
While the first estimate is sharp, the second one isn't and we improve it here. We also give an upper bound of $H_3(1)$ for the class $R_1$.

\medskip

For the study we use a different approach than the common one. In the current research on the upper bound of the Hankel determinant dominates a method based on a result on coefficients of Carath\'{e}odory functions
(functions from with positive real part on the unit disk)  involving Toeplitz determinants. This result is due to Carath\'{e}odory and Toeplitz (\cite[Theorem 3.1.4, p.26]{DTV-book}) and its proof can be found in Grenander and Szeg\H{o} (\cite{granader}).

\medskip

In this paper we use different method, based on the estimates of the coefficients of  Schwartz functions. Here, it is a part of that result needed for the proofs.

\begin{lem}\label{lem-prok}
Let $\omega(z)=c_{1}z+c_{2}z^{2}+\cdots $ be a Schwarz function. Then, for any real numbers $\mu$ and $\nu$ such that $ (\mu,\nu) \in D_1\cup D_2$, where
\[ D_1 = \left\{(\mu,\nu):|\mu|\le\frac12,\, -1\le\nu\le1 \right\}\]
and
\[ D_2 = \left\{(\mu,\nu):\frac12\le|\mu|\le2,\, \frac{4}{27}(|\mu|+1)^3-(|\mu|+1)\le\nu\le1 \right\}, \]
the following sharp estimate holds
$$\left|c_{3}+\mu c_{1}c_{2}+\nu c_{1}^{3}\right|\leq 1.$$
\end{lem}

We will also use the following,  almost forgotten result of Carleson (\cite{carlson}).

\begin{lem}\label{lem-carl}
Let $\omega(z)=c_{1}z+c_{2}z^{2}+\cdots $ be a Schwarz function. Then
\[|c_2|\le1-|c_1|^2 \quad\mbox{and}\quad |c_4|\le1-|c_1|^2 -|c_2|^2. \]
\end{lem}

\medskip

\section{Main results}

\smallskip

First we give the sharp estimate of the third Hankel determinant for the class $\R$.

\begin{thm}\label{main-thm}
Let $f\in\R$ is of the form $f(z)=z+a_2z^2+a_3z^3+\cdots$. Then
\[ |H_3(1)| \le \frac{207}{540}=0.38333\ldots .\]
\end{thm}

\begin{proof}
The condition \eqref{e1} is equivalent to
\[
f'(z) = \frac{1+\omega(z)}{1-\omega(z)},
\]
i.e.
\be\label{e4}
 f'(z) [1-\omega(z)] = 1+\omega(z) ,
\ee
where $\omega$ is analytic in $\D$, $\omega(0)=0$ and $|\omega(z)|<1$ for all $z$ in $\D$. If
\[  \omega(z)= c_1z+c_2z^2+\cdots,  \]
then, by equating the coefficients in \eqref{e4}, we have
\be\label{e6}
\begin{split}
a_2 &= c_1,\\
a_3 &= \frac23(c_1^2+c_2),\\
a_4 &= \frac12(c_3+2c_1c_2+c_1^3),\\
a_5 &= \frac25(c_4+2c_1c_3+3c_1^2c_2+c_1^4+c_2^2).
\end{split}
\ee
Using \eqref{e6} we have
\[
\begin{split}
H_3(1) &=
\frac{1}{540}  \left(-12  c_1 ^4  c_2 -16  c_2 ^3-54  c_1 ^3  c_3 +108  c_1   c_2   c_3 \right.\\
&\quad \left.-135  c_3 ^2+  60 c_1^2 c_2 ^2 -7  c_1 ^6 - 72c_1^2 c_4 +144  c_2   c_4 \right) \\
&=\frac{1}{540}  \Bigg[ -54c_3 \left(c_3-2c_1c_2+c_1^3\right) -81 c_3 ^2 -12  c_1 ^4  c_2 \\
&\quad -16  c_2 ^3  +  60 c_1^2 c_2 ^2 -7  c_1 ^6 + 72(2c_2-c_1^2)c_4 \Bigg],
\end{split}
\]
and
\be\label{e7}
\begin{split}
|H_3(1)| &\le
\frac{1}{540}  \Bigg[54|c_3| \left|c_3-2c_1c_2+c_1^3\right| +81  |c_3| ^2 + 12  |c_1| ^4 | c_2| \\
&\quad +16 | c_2| ^3  +  60| c_1|^2 |c_2| ^2 + 7  |c_1| ^6 +72\left(2|c_2|+|c_1|^2\right)|c_4| \Bigg].
\end{split}
\ee

If we apply  $\left|c_3-2c_1c_2+c_1^3\right| \le1$ (the case when $\mu=-2$, $\nu=1$ and $(\mu,\nu)\in D_2$ in Lemma \ref{lem-prok}),  from \eqref{e7} we get
\be\label{e8}
\begin{split}
|H_3(1)| &\le
\frac{1}{540}  \Bigg[54|c_3|+81|c_3| ^2  +12  |c_1| ^4 | c_2| +16 | c_2| ^3 \\
&\quad  +  60| c_1|^2 |c_2| ^2 + 7  |c_1| ^6 + 72\left(2|c_2|+|c_1|^2\right)|c_4| \Bigg].
\end{split}
\ee

Assume that $|c_2|\leq \tfrac12 (1-|c_1|^2)$. Hence, $2|c_2|+|c_1|^2\leq 1$. From this inequality and Lemma \ref{lem-carl},
\[
\begin{split}
|H_3(1)| &\le
\frac{1}{540}  \Bigg[54|c_3|+81|c_3| ^2  +12  |c_1| ^4 | c_2| +16 | c_2| ^3  +  60| c_1|^2 |c_2| ^2 \\
&\quad +7  |c_1| ^6 + 72\left(1-|c_1|^2-|c_2|^2\right) \Bigg]\ ,
\end{split}
\]
so
\[
\begin{split}
|H_3(1)| &\le
\frac{1}{540}  \Bigg[ 72+54|c_3|+81|c_3|^2+16|c_2|^2(|c_2|-1)+56|c_2|^2(|c_1|^2-1) \\
&\quad  +4|c_1|^2(|c_2|^2-1)+ 7|c_1|^2(|c_1|^4-1)+12|c_1|^2(|c_1|^2|c_2|-1)-49|c_1|^2\Bigg] \\
&\le \frac{1}{540}  ( 72 +54|c_3| +81|c_3|^2),
\end{split}
\]
since all other terms are less or equal to zero.

Providing that $|c_2|\leq \tfrac12 (1-|c_1|^2)$ we have
\[  |H_3(1)| \le  \frac{207}{540}=0.38333\ldots.   \]

Now, assume that $\tfrac12 (1-|c_1|^2)<|c_2|\leq (1-|c_1|^2)$. Applying Lemma \ref{lem-carl} in (\ref{e8}),
\[
\begin{split}
|H_3(1)| &\le
\frac{1}{540}  \Bigg[54|c_3|+81|c_3| ^2  +12  |c_1| ^4 | c_2| +16 | c_2| ^3  +  60| c_1|^2 |c_2| ^2 \\
&\quad +7  |c_1| ^6 + 72(2|c_2|+|c_1|^2)\left(1-|c_1|^2-|c_2|^2\right) \Bigg] .
\end{split}
\]
From our assumption it follows that $2|c_2|+|c_1|^2> 1$, so
\[
\begin{split}
|H_3(1)| &\le  54|c_3| + 81|c_3| ^2  + h_1(|c_1|^2,|c_2|),
\end{split}
\]
where
\[h_1(x,y) = 7x^3-72x^2+72x+12x^2y-12xy^2-144xy-128y^3+144y\ ,\]
$(x,y)\in D$ and $D$ is such that $x+2y>1$, $x+y\leq 1$  and $x\geq 0$.

But $-12xy^2\leq 0$ and $7x^3\leq 7$, so
\[h_1(x,y) < g_1(x,y) = -128y^3+(144-144x+12x^2)y-72x^2+72x+7 .\]
It is enough to derive the greatest value of $g_1$ (even in the square $[0,1]\times[0,1]$). The critical points of $g_1$ satisfy the system of equations
\[\begin{cases}
(x-6)y+3-6x = 0\\
-32y^2+(12-12x+x^2) = 0\ .
\end{cases}\]
The first equation is contradictory if $x\in(1/2,1]$. Suppose that $x\in[0,1/2]$. From this equation $y=\frac{6x-3}{x-6}$. Putting it into the second one we obtain
\[12-12x+x^2-32\left(\frac{6x-3}{x-6}\right)^2 = 0\ ,\]
or equivalently
\[144+480x(1-2x)+6x(1-4x^2)+90x+x^4 = 0\ ,\]
which has no solutions in $[0,1/2]$.

On the boundary of the square $[0,1]\times[0,1]$ there is
\[
\begin{split}
g_1(x,0) &=7+72x-72x^2\leq 25\ , \\
g_1(x,1) &=23-72x-60x^2\leq 23\ ,\\
g_1(1,y) &=7+12y-128y^3\leq 7+\sqrt2\ ,\\
g_1(0,y) &=7+144y-128y^3\leq 7+24\sqrt6=65.787\ldots\ .
\end{split}
\]
This means that in this case,
\[H_3(1) \leq \frac{1}{540} \left(135+65.787\ldots\right) <\frac{207}{540}\ .\]

Summing up, $|H_3(1)| \le  \frac{207}{540}$.
\end{proof}

\medskip

Now we give the estimate of the third Hankel determinant for the class $\R_1$.

\medskip

\begin{thm}\label{th2}
Let $f\in\R_1$ and is of the form $f(z)=z+a_2z^2+a_3z^3+\cdots$. Then
\[ |H_3(1)| \le \frac{3537}{129600}= 0.02729\ldots\ . \]
\end{thm}

\medskip

\begin{proof}
Similarly as in the proof of the previous theorem, for each function $f$ from $\R_1$, there exists a function  $\omega(z)= c_1z+c_2z^2+\cdots$, analytic in $\D$, such that $|\omega(z)|<1$ for all $z$ in $\D$, and
\be\label{eeq}
f'(z)+zf''(z) = \frac{1+\omega(z)}{1-\omega(z)},
\ee
i.e.,
\[  [f'(z)+zf''(z)] [1-\omega(z)] = 1+\omega(z). \]
Equating the coefficients in the previous expression leads to
\be\label{e6-2}
\begin{split}
a_2 &= \frac{c_1}{2},\\
a_3 &= \frac29(c_1^2+c_2),\\
a_4 &= \frac18(c_3+2c_1c_2+c_1^3),\\
a_5 &= \frac{2}{25}(c_4+2c_1c_3+3c_1^2c_2+c_1^4+c_2^2).
\end{split}
\ee
From here, after some calculations we receive
\[
\begin{split}
H_3(1) &=
\frac{1}{1166400}  \left[-1217c_1 ^6-1140  c_1 ^4  c_2 +  13116 c_1^2 c_2 ^2+7936  c_2 ^3-9234 c_1 ^3  c_3  \right.\\
&\quad \left. +972  c_1   c_2   c_3-18225  c_3 ^2  +2592(8c_2-c_1^2)  c_4 \right] \\
&=\frac{1}{1166400}  \Bigg[-8991 c_3 ^2-9234 c_3\left(c_3-\frac{2}{19}c_1c_2+c_1^3\right) -1140  c_1 ^4  c_2 \\
&\quad  +  13116 c_1^2 c_2 ^2 +7936  c_2 ^3  -1217c_1^6 + 2592(8c_2-c_1^2)c_4 \Bigg],
\end{split}
\]
and further
\[
\begin{split}
|H_3(1)| &\le
\frac{1}{1166400}  \Bigg[8991 |c_3|^2 + 9234 |c_3|\left|c_3-\frac{2}{19}c_1c_2+c_1^3\right| +1140  |c_1| ^4  |c_2| \\
& +  13116 |c_1|^2 |c_2| ^2 +7936  |c_2| ^3 +1217|c_1|^6 + 2592(8|c_2|+|c_1|^2)|c_4| \Bigg].
\end{split}
\]

Now, Lemma \ref{lem-prok} for $\mu=-\frac{2}{19}$, $\nu=1$ and $(\mu,\nu)\in D_1$ gives
 $\left|c_3-\frac{2}{19}c_1c_2+c_1^3\right| \le1$  which implies
\[
\begin{split}
|H_3(1)| &\le
\frac{1}{1166400}  \left[9234 |c_3| + 8991 |c_3|^2 + 1140  |c_1| ^4  |c_2| +7936  |c_2| ^3 \right. \\
&\quad \left. +  13116 |c_1|^2 |c_2| ^2+1217|c_1|^6  + 2592(8|c_2|+|c_1|^2)|c_4|\right].
\end{split}
\]

Assume that $|c_2|\leq \tfrac{21}{32} (1-|c_1|^2)$. Hence, $8|c_2|+|c_1|^2\leq \frac14\left(21-17|c_1|^2\right)$. From this inequality and Lemma \ref{lem-carl},
\[
\begin{split}
|H_3(1)| &\le
 \frac{1}{1166400}  \left[ 9234 |c_3| + 8991 |c_3|^2 + 1140  |c_1| ^4  |c_2| +7936  |c_2| ^3\right. \\
&\quad \left.  +  13116 |c_1|^2 |c_2| ^2+1217|c_1|^6  + 648(21-17|c_1|^2)\left(1-|c_1|^2-|c_2|^2\right)\right]
\end{split}
\]
and
\[
\begin{split}
|H_3(1)| &\le
 \frac{1}{1166400}  \left[13608 + 9234 |c_3| +8991 |c_3|^2 + 7936|c_2|^2(|c_2|-1) \right. \\
&\quad  + 7444|c_1|^2(|c_2|^2-1) + 1140|c_1|^2(|c_1|^2|c_2|-1) + 1217|c_1|^2(|c_1|^4-1) \\
&\quad \left.  + 5672|c_2|^2(|c_1|^2-1) -3807|c_1|^2 - 11016|c_1|^2 \left(1-|c_1|^2-|c_2|^2\right)\right] \\
&\le \frac{1}{1166400}  ( 13608 + 9234 |c_3| +8991 |c_3|^2),
\end{split}
\]
since all other terms are less or equal to zero.

The greatest value of the function in brackets is attained for $|c_3|=1$ and it is equal to 31833. In this way we have proven that
\[  |H_3(1)| \le  \frac{31833}{11664000} = \frac{3537}{129600} = 0.02729\ldots   \]
under the condition $|c_2|\leq \tfrac{21}{32} (1-|c_1|^2)$.

Assume now that $\tfrac{21}{32} (1-|c_1|^2) <|c_2|\leq (1-|c_1|^2)$. From Lemma \ref{lem-carl},
\[
\begin{split}
|H_3(1)| &\le
\frac{1}{1166400}  \left[9234 |c_3| + 8991 |c_3|^2 + 1140  |c_1| ^4  |c_2| +7936  |c_2| ^3 +  13116 |c_1|^2 |c_2| ^2\right. \\
&\quad \left. +1217|c_1|^6  + 2592(8|c_2|+|c_1|^2)\left(1-|c_1|^2-|c_2|^2\right)\right].
\end{split}
\]

From the assumption it follows that $8|c_2|+|c_1|^2 > \frac14\left(21-17|c_1|^2\right)$ and
\[
\begin{split}
|H_3(1)| &\le
\frac{1}{1166400} \left[ 9234|c_3|+8991|c_3| ^2  + h_2(|c_1|^2,|c_2|)\right] \\
&\leq \frac{1}{1166400} \left[ 18225  + h_2(|c_1|^2,|c_2|)\right] \,
\end{split}\]
where
\[h_2(x,y) = -12800y^3+10524xy^2+(1140x^2-20736x+20736)y+1217x^3-2592x^2+13608x\]
and $(x,y)\in D$, $D$ is such that $x+\tfrac{21}{32}y>1$, $x+y\leq 1$ and $x\geq 0$.

We shall derive the greatest value of $h_2$ in $E=\{(x,y):x\geq 0, y\geq 0, x+y\leq 1\}$, i.e. in the superset of $D$. Note that
\[\begin{split}
\frac{\partial h_2}{\partial x}
&= 3\left(3508y^2-6912y+760xy+1217x^2-1728x+4536\right) \\
&= 3\left[760(1-x)(1-y)+3076(1-y)^2+484(1-x)^2+216+733x^2+432y^2\right] \\
&\ge 0.
\end{split}\]
It means that the greatest value of $h_2$ is obtained on the boundary of $E$. We have
\[
\begin{split}
g_2(x,0) &=1217x^3+11016x+2592x(1-x)\leq 1217x^3+11016x\leq 12233\ , \\
g_2(0,y) &=20736y-12800y^2\leq \frac{209952}{25} = 8398.08\ldots \ .
\end{split}
\]
Additionally, it is not difficult to show that
\[g_2(x,1-x) =7936+21060x-40164x^2+23401x^3 \leq 12233\ .\]
Hence, in this case,
\[H_3(1) \leq \frac{1}{1166400} \left(18225+12233\right) = \frac{15229}{583200} = 0.02611\ldots \ .\]

Summing up, $|H_3(1)| \le  \frac{3537}{129600}=0.02729\ldots$\ .
\end{proof}

\end{document}